\documentclass[12pt]{amsart}
\usepackage{amsthm,amsmath,amsfonts,amssymb,latexsym,amscd,mathrsfs,hyperref,cleveref,graphicx}
\usepackage{setspace}
\setdisplayskipstretch{2}
\usepackage{cite}

\theoremstyle{plain}
\newtheorem{thm}{Theorem}
\newtheorem{lem}{Lemma}

\theoremstyle{definition}

\newtheorem*{exa*}{Example}

\newtheorem{case}{Case}
 
\crefname{thm}{Theorem}{theorems}
\crefname{lem}{Lemma}{lemmas}

\newcommand{\ds}{\displaystyle}
\newcommand{\bs}{\boldsymbol}
\newcommand{\mb}{\mathbb}
\newcommand{\mc}{\mathcal}

\renewcommand{\mod}{\operatorname{mod}}

\def \a{\alpha} \def \b{\beta} \def \d{\delta} \def \e{\varepsilon} \def \g{\gamma} \def \k{\kappa} \def \l{\lambda} \def \s{\sigma} \def \t{\theta} \def \z{\zeta}

\numberwithin{equation}{section}

\renewcommand{\labelenumi}{\setlength{\labelwidth}{\leftmargin}
   \addtolength{\labelwidth}{-\labelsep}
   \hbox to \labelwidth{\theenumi.\hfill}}

\begin{document}
\title{Fractional parts of polynomials over the primes. II}
\author[Roger Baker]{Roger Baker$^\dag$}
\thanks{$^\dag$Research supported in part by Collaboration Grant 412557 from the Simons Foundation}

\address{Department of Mathematics\newline
\indent Brigham Young University\newline
\indent Provo, UT 84602, U.S.A}
\email{baker@math.byu.edu}

\dedicatory{\rm To Glyn Harman on his sixtieth birthday.}

 \begin{abstract}
Let $\|\ldots\|$ denote distance from the integers. Let $\a$, $\b$, $\g$ be real numbers with $\a$ irrational. We show that the inequality
 \[
\|\a p^2 + \b p+ \g\| < p^{-37/210}
 \]
has infinitely many solutions in primes $p$, sharpening a result due to Harman (1996) in the case $\b = 0$ and Baker (2017) in the general case.
 \end{abstract}
 


\maketitle

\section{Introduction}\label{sec1}

Let $f_k(x) = \a x^k + \cdots + \b x + \g$ be a polynomial of degree $k > 1$ with irrational leading coefficient. Inequalities of the form
 \begin{equation}\label{eq1.1}
\|f_k(p)\| < p^{-\rho_k}
 \end{equation}
for infinitely many primes $p$ were studied by Vinogradov \cite{vin}; see \cite{bak} for the strongest available results. The present paper gives a new result for $k=2$.

 \begin{thm}\label{thm1}
Let $\rho_2 = 37/210 = 0.1761\ldots$; then \eqref{eq1.1} holds for infinitely many primes $p$.
 \end{thm}

The values $\rho_2 = 1/8 - \e$, $3/20$, $2/13 = 0.1538\ldots$ were given by Ghosh \cite{ghosh}, Baker and Harman \cite{bakhar}, and Harman \cite{har} in the case $\b=0$; Baker \cite{bak} extended the last result to general $\b$.

As in \cite{bak}, \cite{bakhar}, \cite{har} we use the Harman sieve. We make progress in the present paper by giving new bounds for sums of the shape
 \[
\sum_{\ell \le L} c_\ell \sum_{R < r \le 2R} a_r \sum_{\frac N2 < sr \le N} b_s e(\ell g(rs))
 \]
where $g$ is the approximating polynomial to $f_2$ in \cite{bak}. Type I sums (in which $b_s \equiv 1$) are treated in Section 2, and general (Type II) sums in Section 3. In the Type I case, $(a_r)_{R < r \le 2R}$ is restricted to convolutions of shorter sequences; a lemma of Birch and Davenport \cite{birdav} on Diophantine approximation plays a key role. For Type II sums, a subsidiary task is the study of the average behavior, as $n$ varies, of the number of solutions of
 \[
\ell_1 y_1^2 - \ell_2 y_2^2 = n \quad (\ell_1, \ell_2 \le L; \ y_1, y_2 \in [Y, 2Y)).
 \]

One sum that eludes the Harman sieve takes the form
 \[
\sum_{\substack{(2N)^{\frac 12 - 9\s} \le p < (2N)^{\frac 38 + \frac{33\s}4}\\
pp' \in \left(\frac N2, N\right], \|g(pp')\|\le\d}} 1
 \]
(where $\rho = \frac{37}{210} + \e$, $\s = \rho - \frac 16$ and $\d = \frac 12\, N^{-\rho+\frac \e2}$). We bound it above using the form of the linear sieve given by Iwaniec \cite{iwan}. This sum is treated in Section 5; the sums accessible via the Harman sieve are in Section 4. Section 6 contains the sieve decomposition of $S(\mc A, (2N)^{1/2})$ (defined below), and the calculations leading to \cref{thm1}. Integrals that appear here and in earlier drafts were calculated by Andreas Weingartner; thanks, Andreas, for your generosity.

The following notations will be used:\\[.5mm]
 
\noindent $\|\t\| = \min\limits_{n \in \mb Z} |\t - n|$.\\

\noindent $|\mc E| = \sum\limits_{n \in\mc E} 1 \qquad (\mc E \subset [1, N])$.\\

\noindent $\chi_{\mc E} =$ indicator function of $\mc E$.\\

\noindent $C \ldots$ absolute constant, not the same at each occurrence.\\

\noindent $\l \ldots$ real number with $|\l| \le C$, not the same at each occurrence.\\

\noindent $\e\ldots$ sufficiently small positive number; $\eta = \e^9$.\\

\noindent $\gg, \ll\ldots$ indicate implied constants that may depend on $\e$.\\

\noindent $y \sim Y \ldots$ indicates $Y < y \le 2Y$.\\

\noindent $e(\t) = e^{2\pi i\t}$.\\

\noindent $\frac aq \ldots$ fraction in lowest terms with $|\a_2 - \frac aq|< \frac 1{q^2}$, with $q$ sufficiently large; $g(x) = \frac aq\, x^2 + \b x + \g$.
 \bigskip

We choose $N$ so that
 \[
L_1^{1/2} N \ll q \ll L_1^{1/2} N \ \text{ with } \ L_1 = 2N^{\rho-\frac \e2}
 \]
and write
 \[
L = N^{\rho-\e/3}.
 \]
Here $\rho$ will ultimately be $\frac{37}{210} + \e$; earlier in the paper we restrict $\rho$ somewhat less. We do suppose $\rho > \frac 16$, and write
 \[
\s = \rho - 1/6.
 \]

We reserve the symbols $p, p_1, p', \ldots$ for prime numbers.

Let $\d = L_1^{-1}$ and
 \[
\mc A = \left\{n : n \sim \frac N2, \|g(n)\| < \d\right\}, \ \mc B = \left\{n : n \sim \frac N2\right\}.
 \]
We write $I(m)$ for an arbitrary subinterval of $\left(\frac N{2m}, \frac Nm\right]$.
 \bigskip

\section{Type I sums}\label{sec2}

The object of this section is to prove

 \begin{thm}\label{thm2}
Let $\frac 16<\rho<\frac 2{11}$. Let $V\ge 1, W\ge 1$, 
 \[
V^3W^2\ll N^{2-3\rho}, VW^3\ll N^{3-15\rho/2}, VW\ll N^{2-8\rho}.
 \]
Let
 \[
T: = \sum_{\ell = 1}^L \sum_{v\sim V} \sum_{w\sim W}\left|\sum_{n\epsilon I(vw)} e(\ell g(vwn))\right|.
 \]
Then
 \[
T \ll N^{1-10\eta}.
 \]
 \end{thm}
We require several lemmas.

 \begin{lem}\label{lem1}
Let $1\le Y\ll N^{1-2\rho}$. Let $\mc S$ be the set of $y\in (Y,4Y]$ with $\langle y,q\rangle \le N^\rho$ and
 \begin{equation} \label{eq 2.1}
T(y): = \sum_{\ell =1}^L \left| \sum_{n\in I(y)} e(\ell g(yn))\right| >N^{1-10\eta} Y^{-1}.
 \end{equation}
There are a set $\mc S^* \subset \mc S$ and positive numbers $S$, $Z$ with the following properties.

 \begin{enumerate}
\item[(i)] We have
 \begin{equation}\label{eq2.2}
\sum_{y\in \mc S} T(y)\ll L^{1/2}N^{10\eta}Z^{1/2}\left| \mc S^*\right|;
 \end{equation}
\item[(ii)]
for $y\in \mc S^*$, we have
 \begin{equation}\label{eq2.3}
\left|\frac{sa\, y^2}q-u\right| <Z^{-1}
 \end{equation}
for some $s=s(y)$ in $\mb N, u=u(y)$ in $\mb Z, (u,q)=1$, and
 \begin{equation}\label{eq2.4}
s\sim S\ll L^2 N^\eta;
 \end{equation}
\item[(iii)] $Z$ satisfies
 \begin{equation}\label{eq2.5}
(N/Y)^2 L^{-1}N^{-\eta}\ll Z \ll LS^{-1}(N/Y)^2.
 \end{equation}
 \end{enumerate}
 \end{lem}

\begin{proof}
This can readily be extracted from the proof of \cite [Lemma 8]{bak}, with $5\eta$ in place of $\eta$.
\end{proof}

\begin{lem}\label{lem2}
Let $\t$ be a real number and suppose there exist $R$ distinct integer pairs $x,z$ satisfying
 \begin{equation}\label{2.6}
\left|\t x-z\right|<\z ,\ 0 <|x|<X,
 \end{equation}
where $R\ge 24\z X>0$. Then all integer pairs $x$, $z$ satisfying \eqref{2.6} have the same ratio $z/x$.
\end{lem}

 \begin{proof}
Birch and Davenport\cite{birdav}.
 \end{proof}

 \begin{lem}\label{lem3}
Suppose that $s\ge 1$, $D\ge 1$, $sD^2<q$, $Z\ge 2$. The number of solutions $y\in(Y,4Y]$, with $\langle y,q\rangle \le D$, of the inequality
 \[
\left\| \frac{sa\,y^2}{q}\right\| <\frac 1Z
 \]
is
 \[
\ll N^\eta \left(\frac{Y+q^{1/2}}{Z^{1/2}}\right).
 \]
 \end{lem}

 \begin{proof}
\cite[Lemma 3]{bak}.
 \end{proof}

 \begin{proof}[Proof of \cref{thm2}]
In the notation of \Cref{lem1}, with $Y=VW$, let $\mc C$ be the set of pairs $(v,w)$ for which $v\sim V$, $w\sim W$, $ vw\in \mc S^*$. As in \cite{bak}, proof of Lemma 8, it suffices to prove that
 \begin{equation}\label{eq2.7}
\sum_{(v,w)\in \mc C} T(v,w) \ll N^{1-10\eta}.
 \end{equation}

In view of the Type I result obtained in \cite[Lemma 8]{bak}, we may suppose that
 \[
VW>N^{1-5\rho/2-\eta}.
 \]
We note that
 \[\left| \mc C \right| \ll \frac{N^{2\eta}S(VW+q^{1/2})}{Z^{1/2}}
 \]
as a consequence of \Cref{lem3}; indeed
 \begin{equation}\label{eq2.8}
\left| \mc C \right|\ll \frac{N^{2\eta}SVW}{Z^{1/2}},
 \end{equation}
since $VW>N^{1-5\rho/2-\eta}>N^{1/2+\rho/4}\gg q^{1/2}$ for $\e$ sufficiently small.
\bigskip

Suppose for a moment that
 \begin{equation}\label{eq2.9}
S<N^{1-23\eta}(VW)^{-1}L^{-1/2};
 \end{equation}
then \eqref{eq2.2} gives (with a divisor argument)
 \begin{align*}
\sum_{(v,w)\in \mc C} T(v,w) &\ll L^{1/2} N^{11\eta}Z^{1/2} | \mc C |\\
&\ll N^{1-10\eta}
 \end{align*}
(from \eqref{eq2.8},\eqref{eq2.9}), giving \eqref{eq2.7}. So we may suppose that
 \[
S\ge N^{1-23\eta}(VW)^{-1}L^{-1/2}.
 \]

It now follows from \eqref{eq2.5} that
 \begin{equation}\label{eq2.10}
Z\le \left(\frac{N}{VW}\right)^2 L\, \frac{VWL^{1/2}}{N^{1-23\eta}}=\frac{N^{1+23\eta}}{VW}\, L^{3/2}.
 \end{equation}

Now let
 \[
\mc C(w)=\{v\sim V: (v,w)\in \mc C\}
 \]
and for $K\geq$ 1, let
 \[
\mc E (K)=\{ w\sim W:K\le |\mc C(w)|<2K\}.
 \]
Then
 \begin{align*}
|\mc C |&=\sum_{w\sim W}|\mc C(w)|\\
&\le \sum_{K=2^m\in [1,V]}2K|\mc E(K)|.
 \end{align*}
We choose $K,1\le K\le V$, so that
 \begin{equation}\label{eq2.11}
|\mc C| \ll \log N |\mc E(K)| K.
 \end{equation}

Suppose for a moment that
 \[
L^{1/2} Z^{1/2} KW<N^{1-22\eta}.
 \]
Then arguing as above,
 \begin{align*}
\sum_{(v,w)\in \mc C}T(v,w)&\ll L^{1/2} N^{11\eta}Z^{1/2} \log N|\mc E(K)|K\\
&\ll N^{12\eta} L^{1/2} Z^{1/2} WK \ll N^{1-10\eta}.
 \end{align*}
Thus we may suppose that
 \begin{equation}\label{eq2.12}
K\ge \frac{N^{1-22\eta}}{L^{1/2}Z^{1/2}W}. 
 \end{equation}

For the next stage of the argument, let $w$ be a fixed integer in $\mc E(K)$. We apply \cref{lem2}, taking
 \[
\t =\frac{w^2a}{q},\ X=8SV^2,\ \zeta =\frac 1Z
 \]
since (by \eqref{eq2.3} and the definition of $\mc C$)
 \[
|(sv^2)\t-u|<\frac 1Z\ ,\ s=s(v,w), u=u(v,w)
 \]
for every $v$ in $\mc C(w)$. By a divisor argument, the number of distinct $sv^2$ as $v$ varies over $\mc C(w)$ is $\gg KN^{-\eta}$. Thus in the notation of \cref{lem2},
 \[
R\gg KN^{-\eta}\gg X\zeta N^\eta=\frac{8SV^2N^\eta}Z.
 \]
To see this,
 \begin{align*}
KN^{-\eta}(8SV^2N^\eta/Z)^{-1}&\gg \frac{N^{1-24\eta}}{Z^{1/2}L^{1/2}W}\ \frac Z{SV^2}\\
\intertext{(from \eqref{eq2.12})}
&\gg \frac{N^{1-25\eta}Z^{1/2}}{L^{5/2}V^2W}\gg \frac{N^{1-26\eta}(N/VW)}{L^3V^2W}\\
\intertext{(from \eqref{eq2.4})}
&\gg1
 \end{align*}
from the hypothesis of the theorem.

Accordingly, all $\frac u{sv^2}$ with $v\in \mc C(w)$ can be written in the form
 \begin{equation}\label{eq2.13}
\frac u{sv^2}=\frac tr\ ,\ (t,r)=1,\ t=t(w),\ r=r(w)
 \end{equation}
for a certain $t\in \mb Z$, $r\in \mb N$ independent of $v$.

We record a lower bound for $K$ that does not contain $Z$. From \eqref{eq2.12}, \eqref{eq2.10}, we have
 \begin{equation}\label{eq2.14}
K\ge \frac{N^{1-22\eta}}{L^{1/2}W}\ \frac{(VW)^{1/2}}{L^{3/4}N^{1/2+12\eta}}=N^{\frac 12-34\eta}V^{1/2}W^{-1/2}L^{-5/4}.
 \end{equation}

We now select a divisor $z$ of $r$ such that the set
 \[
\mc C(w,z)=\left\{v\in \mc C(w):\frac r{\langle s,r \rangle}=z\right\}
 \]
satisfies
 \begin{equation}\label{eq2.15}
KN^{-\eta}\ll |\mc C(w,z)|\le 2K.
 \end{equation}
For each $v$ in $\mc C(w,z)$, we have
 \[
z\mid v^2.
 \]
It is convenient to write $z=bc^2$ where $b$ is square-free, and define $k=v^2/bc^2$. Then
 \[
bc^2k=v^2\ ,\ k=bd^2\ \text{(where $d\in \mb N$)},\ bcd=v.
 \]

\noindent This leads to the upper bound
 \[
|\mc C(w,z)|\ll \frac V{bc},
 \]
from which we infer, using \eqref{eq2.15}, that
 \begin{equation}\label{eq2.16}
bc\ll \frac{VN^\eta}K.
 \end{equation}
Now we re-examine our rational approximation
 \[
\left|\frac{sav^2w^2}q-w\right|<Z^{-1}
 \]
(see \eqref{eq2.3} and the definition of $\mc C$). We observe that
 \begin{align}
\left|r\frac{w^2a}q-t\right| &=\frac r{sv^2}\left|\frac{sv^2w^2a}q-u\right|\label{eq2.17}\\[2mm]
&\le \frac z{V^2}\ \frac 1Z=\frac{bc^2}{V^2}\ \frac 1Z\ll \frac{N^{2\eta}}{K^2Z}\notag
 \end{align}
for $v\in \mc C(w,z)$, using \eqref{eq2.16}.

Let $\mc F$ be the set of natural numbers $b_0c^2_0m$, $(b_0, c_0,m)\in \mb N^3$, $b_0c_0<\frac{VN^{2\eta}}K$, $m<L^2N^\eta$.
The number of possibilities for $b_0c_0^2$ here is $\ll \frac{VN^{3\eta}}K$, since $b_0c_0^2$ is a divisor of $(b_0c_0)^2$. Hence
 \begin{equation}\label{eq2.18}
|\mc F|\ll \frac{VL^2N^{4\eta}}K.
 \end{equation}
All the integers $r$ occurring in \eqref{eq2.17} are in $\mc F$. Hence
 \begin{align}
|\mc E(K)|&=\sum_{r\in \mc F}\ \sum_{w\in \mc E(K),r(w)=r}1\label{eq2.19}\\[2mm]
&\ll |\mc F|\frac{N^{2\eta}(W+q^{1/2})}{KZ^{1/2}}\notag,
 \end{align}
on bounding the number of $w$ in $\mc E(K)$ with $r(w)=r$ via \cref{lem3}. Combining \eqref{eq2.18}, \eqref{eq2.19}, and recalling \eqref{eq2.11},
 \begin{align*}
|\mc E(K)| &\ll \frac{VL^2N^{6\eta}}K\ \frac{(W+q^{1/2})}{KZ^{1/2}},\\[2mm]
\sum_{(v,w)\in \mc C} T(vw)&\ll L^{1/2}N^{11\eta}Z^{1/2}|\mc C|\\[2mm]
&\ll L^{1/2}N^{12\eta}Z^{1/2}|\mc E(K)| K\\[2mm]
&\ll \frac{L^{5/2}N^{18\eta}V(W+q^{1/2})}K.
 \end{align*}
We now use the lower bound \eqref{eq2.14} for $K$ and obtain
 \[
\sum_{(v,w)\in \mc C} T(vw)\ll L^{15/4}(VW)^{1/2}N^{-\frac 12+60\eta}(W+q^{1/2}).
 \]
Now \eqref{eq2.7} follows on applying the bounds for $VW^3$ and $VW$ in the hypothesis of \cref{thm2}. This completes the proof of \cref{thm2}.
 \end{proof}
 \bigskip

\section{Type II sums}\label{sec3}

\begin{lem}\label{lem4}
Let $Y\ge 1$. For $n\in \mb N$, let $R(n)$ denote the number of quadruples $(\ell_1,\ell_2,y_1,y_2)$ with $1\le \ell_i \le L$, $y_i\sim Y$ such that
 \begin{gather}
\ell_1y_1^2-\ell_2y_2^2=n.\label{eq3.1}\\
\intertext{Then}
R(n)\ll(LY)^{1+\eta},\label{eq3.2}\\[2mm]
\sum_{n\in \mb Z}R(n)^2\ll L^{3+\eta}Y^{2+\eta}.\label{eq3.3}
 \end{gather}
 \end{lem}

 \begin{proof}
For \eqref{eq3.2}, fix $\ell_2$ and $y_2$, then the equation
 \[
\ell_1y_1^2=\ell_2y_2^2+n
 \]
(with $\ell_1\neq 0$, $y_1\neq 0$) determines $\ell_1$, $y_1$ up to $O((LY)^\eta)$ possibilities.

For \eqref{eq3.3}, we observe that
 \[
\sum_{n\in \mb Z}R(n)^2
 \]
is the number of tuples $\ell_1$, $\ell_2$, $\ell_3$, $\ell_4$, $y_1$, $y_2$, $y_3$, $y_4$ with $1\le \ell_i\le L$, $y_i\sim Y$ and
 \[
\ell_1y_1^2-\ell_2y_2^2=\ell_3y_3^2-\ell_4y_4^2.
 \]
This may be expressed as an integral:
 \begin{align}
\sum_{n \in \mb Z}R(n)^2&=\int_0^1\left|\sum_{\ell \le L}\ \sum_{y\sim Y}e(\ell y^2t)\right|^2 \left|\sum_{\ell_0\le L}\ \sum_{y_0\sim Y}e(\ell_0y_0^2)\right|^2 dt\label{eq3.4}\\[2mm]
&\le LV\notag
 \end{align}
by the Cauchy-Schwarz inequality, where
 \[
V=\sum_{1\le \ell \le L}\int_0^1\left|\sum_{y\sim Y}e(\ell y^2t)\right|^2\ \sum_{\ell_0\le L}\ \sum_{y_0\sim Y}e(\ell_0y_0^2t)^2dt
 \]
Now $V$ is the number of solution of 
 \begin{equation}\label{eq3.5}
\ell(y_1^2-y_2^2)=\ell_3y_3^2-\ell_4y_4^2
 \end{equation}
with $1\le \ell$, $\ell_3$, $\ell_4\le L$ and $y_i\sim Y(1\le i\le 4)$.

We first consider $V_1$, the number of solutions of \eqref{eq3.5} with
 \begin{equation}\label{eq3.6}
\ell_3y_3^2=\ell_4y_4^2.
 \end{equation}
If \eqref{eq3.5} and \eqref{eq3.6} hold, then $y_1=y_2$. There are $O((LY)^{1+\eta})$ possibilities for $\ell_3$, $\ell_4$, $y_3$, $y_4$ and for each of these, at most $LY$ possibilities for $\ell$, $y_1$, $y_2$. Thus
 \[
V_1\ll (LY)^{2+\eta}.
 \]

Now consider $V_2$, the number of solutions of \eqref{eq3.5} for which \eqref{eq3.6} is violated. There are $O(L^2Y^2)$ possibilities for $\ell_3$, $y_3$, $\ell_4$, $y_4$. For each of these, there are $O((LY)^\eta)$ possibilities for $\ell$, $y_1-y_2$ and $y_1+y_2$, hence $O((LY)^\eta)$ possibilities for $\ell$, $y_1$, $y_2$. Thus
 \begin{equation}\label{eq3.7}
V_2\ll (LY)^{2+\eta}, V\ll(LY)^{2+\eta}.
 \end{equation}

Now\eqref{eq3.3} now follows on combining\eqref{eq3.4} and \eqref{eq3.7}.
 \end{proof}

 \begin{thm}\label{thm3}
For $\frac 16<\rho<\frac 15$, and $N^\rho\ll Y\ll N^{1-4\rho}$, $|c_\ell|\le 1$, $|a_x|\le 1$, $|b_y|\le 1$, we have
 \[
\sum_{l\le L}c_\ell\ \sum_{y\sim Y}b_y \ \sum_{x\in I(y)}a_x e(\ell g(xy))\ll N^{1-10\eta}.
 \]
 \end{thm}
 
 \begin{proof}
Just as in \cite[proof of Lemma 9]{bak} we need only show that
 \begin{equation}\label{eq3.8}
S':=\sum_{\ell\le L}c_\ell\ \sum_{y\sim Y}b_y\ \sum_{x\le \frac NY}a_xe(\ell g(xy))\ll N^{1-11\eta}.
 \end{equation}
Again arguing as in that proof,
 \begin{equation}\label{eq3.9}
|S'|^2\le \frac NY \sum_{\bs\ell \in \mc L} |S(\bs\ell)|,
 \end{equation}
where $\bs \ell =(\ell_1, \ell_2, y_1, y_2)$, $\mc L=\{\bs \ell : \ell_1, \ell_2 \le L, y_1, y_2\sim Y\}$ and
 \[
S(\bs \ell):=\sum_{x\le \frac NY}e(\ell_1 g(xy_1)-\ell_2g(xy_2)).
 \]
The contribution to the right-hand side of \eqref{eq3.9} from those $\bs \ell$ with
 \[
\ell_1y_1^2=\ell_2y_2^2
 \]
is
 \begin{align*}
&\ll \frac NY \cdot LY \cdot \frac NY^{1+\eta}\\
\intertext{(by a divisor argument)}
&\ll N^{2-22\eta}
 \end{align*}
since $Y\gg N^\rho$. The contribution from those $\bs \ell$ with
 \begin{align*}
|S(\bs \ell)|&\le \frac NY^{1-22\eta}L^{-2}\\
\intertext{is}
&\le \frac NY. L^2Y^2. \frac NY^{1-22\eta}L^{-2}=N^{2-22\eta}.
 \end{align*}

It remains to consider $\mc M$, the set of $\bs \ell$ in $\mc L$ with
 \[
\ell_1y_1^2>\ell_2y_2^2
 \]
and
 \[
|S(\bs \ell)|> \frac NY^{1-22\eta}L^{-2}.
 \]
We apply \cite[Lemma 5]{bak} with $M=1$, $X = NY^{-1}$, $P=\frac{N^{1-22\eta}}{YL^2}$.
We require
 \[
P\ge X^{\frac 12+\eta}
 \]
which holds since
 \begin{align*}
PX^{-\frac 12-\eta}&\ge N^{\frac 12-23\eta}L^{-2}Y^{-1/2}\\[2mm]
&\ge N^{\frac 12-23\eta}L^{-2}N^{-1/2+2\rho}\ge 1.
 \end{align*}
Thus for each $\bs \ell \in \mc M$ there exists a natural number $s$,
 \begin{gather}
s=s(\bs \ell)\le L^4N^\eta, \left|s(\ell_1y_1^2-\ell_2y_2)\frac aq-u_2\right|=\frac 1{Z(\bs \ell)}\label{eq3.10}\\[2mm]
|s(\ell_1\a_1y_1-\ell_2\a_1y_2)-u_1|=\frac 1{W(\bs \ell)}\notag
 \end{gather}
with $u_1$, $u_2\in \mb Z$, $(u_2,s)=1$,
 \[
Z(\bs \ell)\ge \left(\frac NY\right)^2L^{-4}N^{-\eta}, W(\bs \ell)\ge \frac NY\, L^{-4}N^{-\eta}.
 \]
Now $s\le N/Y$ and, writing $\g_2=(\ell_1y_1^2-\ell_2y_2^2)a/q$, $\g_1=(\ell_1y_1-\ell_2y_2)\a_1,$
 \[
|s\g_j-u_j|\leq(2k^2)^{-1}\left(\frac NY\right)^{1-j}(j=1,2).
 \]
Thus we can appeal to \cite[Lemma 7]{bak} with $k=2$ and $L$ replaced by 1. Let
 \begin{gather*}
\beta_j=\g_j-\frac{u_j}s, F(x)=\sum_{j=1}^2\beta_jx^j \ , \ G(x)=\sum_{j=1}^2u_jx^j.\\[2mm]
S(s,G)=\sum_{u=1}^se\left(\frac{G(u)}s\right).
 \end{gather*}
We obtain
 \begin{equation}\label{eq3.11}
S(\bs\ell)=\sum_{x\le \frac NY}e(\g_2x^2+\g_1x) =s^{-1}S(s,G)\int_0^{N/Y}e(F(z))dz + O(N^{3\eta}L^2).
 \end{equation}
Now $N^{3\eta}L^2$ is of smaller order than $\frac{N^{1-22\eta}}Y\, L^{-2}$, so that
 \begin{equation}\label{eq3.12}
|s^{-1}S(s,G)\int_0^{N/Y}e(F(z))dz|\gg \frac{N^{1-22\eta}}Y\, L^{-2}.
 \end{equation}
Moreover, by standard bounds, the left-hand side of \eqref{eq3.12} is
 \begin{equation}\label{3.13}
\ll S^{-1/2} \min(N/Y, (SZ(\bs\ell))^{1/2}).
 \end{equation}

We now use a standard splitting-up argument to choose a subset $Q$ of $\mc M$ such that
 \begin{equation}\label{eq3.14}
s(\bs\ell)\sim S, \left|s(\ell_1y_1^2-\ell_2y_2^2)\, \frac aq-u_2\right| <\frac 1Z\quad (\bs\ell\in Q),
 \end{equation}
where $(N/Y)^2 L^{-4}N^{-\eta} \le Z \le Z_0$, with
 \[
\frac NY=(SZ_0)^{1/2},
 \]
and moreover
 \[
S\le L^4N^\eta, Z\ge \left(\frac NY\right)^2L^{-4}N^{-\eta},
 \]
while
 \[
\sum_{\bs\ell \in \mc M}|S(\bs\ell)|\le (\log N)^2 \sum_{\bs \ell\in Q}|S(\bs\ell)|.
 \]
Compare e.g. the argument in \cite[proof of Lemma 8]{bak}.
In order to obtain \eqref{eq3.8} it remains to show that
 \[
\frac NY\sum_{\bs\ell\in Q}|S(\bs\ell)|\ll N^{2-23\eta}.
 \]
Using \eqref{eq3.11}--\eqref{eq3.14}, we find that
 \[
|S(\bs\ell)|\ll Z^{1/2}\quad (\bs\ell \in Q)
 \]
and we must show that
\begin{equation}\label{eq3.15}
Z^{1/2}|Q|\ll YN^{1-23\eta}.
\end{equation}

For each $\mc\ell$ in $Q$ there is an $s\sim S$ with
 \[
s(\ell_1y_1^2-\ell_2y_2^2)\in \mc C,
 \]
where
 \[
\mc C=\left\{n:1\le n\le 2SLY^2,\ |na\ (\mod q)|<\frac qZ\right\}.
 \]
Clearly
 \begin{align}
|\mc C|\ll \left(\frac{SLY^2}q+1\right) \left(\frac qZ+1\right) &= \frac{SLY^2}{Z}+\frac{SLY^2}q\label{eq3.16}\\[2mm]
&+\frac qZ+1.\notag
 \end{align}

Given $m\in \mc C$, let $h_1(m),\ldots,h_j(m)$, $j = j(m)\ll N^\eta$, be the divisors $h$ of $m$ with $h \sim S$. Each element of $Q$ satisfies 
 \[
(\ell_1 y_1^2 - \ell_2 y_2^2) h_i(m) = m
 \]
for some $m \in \mc C$ and some $i$, $1 \le i \le j(m)$. Let
 \[
\mc K = \{m/h_i(m) : 1 \le i \le j(m), \ m \in \mc C\}.
 \]
Then
 \begin{equation}\label{eq3.17}
|\mc K| \ll N^\eta |\mc C|, 
 \end{equation}
while
 \begin{align*}
|Q| &= \sum_{\substack{\bs\ell \in \mc M\\
\ell_1 y_1^2 - \ell_2 y_2^2 \in \mc K}} 1\\[2mm]
&\le \sum_{n\in \mc K} \ \sum_{\substack{\bs\ell \in \mc M\\
\ell_1 y_1^2 - \ell_2 y_2^2 = n}} 1\\[2mm]
&= \sum_{n \in \mc K} \ R(n)
 \end{align*}
in the notation of \cref{lem4}. Applying Cauchy's inequality,
 \begin{align}
|Q| &\le |\mc K|^{1/2} \Bigg(\sum_{n\ge 1}\, R(n)^2\Bigg)^{1/2}\label{eq3.18}\\[2mm]
&\ll (L^3 Y^2)^{1/2} |\mc C|^{1/2} N^\eta,\notag 
 \end{align}
by \cref{lem4} and \eqref{eq3.17}.

Alternatively, \eqref{eq3.2} yields
 \begin{equation}\label{eq3.19}
|Q| \le \left(\max_{n \in \mc K} \, R(n)\right) |\mc K| \ll LY |\mc C| N^{2\eta}.
 \end{equation}
 
We now find that, depending in the value of $Y$, either \eqref{eq3.18} or \eqref{eq3.19} yields the desired bound \eqref{eq3.15}. Suppose first that
 \begin{equation}\label{eq3.20}
Y > N^{3\rho/2} S^{-1/2}. 
 \end{equation}
In view of \eqref{eq3.16}, \eqref{eq3.18}, we need to verify the four bounds

 \begin{align}
Z^{1/2} (L^3Y^2)^{1/2} \left(\frac{SLY^2}Z\right)^{1/2} &\ll YN^{1-24\eta},\label{eq3.21}\\[2mm]
Z^{1/2}(L^2Y^2)^{1/2}\left(\frac{SLY^2}q\right)^{1/2} &\ll YN^{1-24\eta},\label{eq3.22}\\[2mm]
Z^{1/2}(L^3Y^2)^{1/2}\left(\frac qZ\right)^{1/2} &\ll YN^{1-24\eta},\label{eq3.23}\\[2mm]
Z^{1/2}(L^3Y^2)^{1/2} &\ll YN^{1-24\eta}.\label{eq3.24}
 \end{align}

First of all, \eqref{eq3.21} holds since
 \[
L^2 Y^2 S^{1/2}(YN^{1-24\eta})^{-1} \ll YN^{4\rho-1} \ll 1.
 \]
Next, \eqref{eq3.22} holds since
 \begin{align*}
Z^{1/2} L^2 & S^{1/2} Y^2 q^{-1/2}(YN^{1-24\eta})^{-1}\\[2mm]
&\ll S^{-1/2} NY^{-1} L^2 YS^{1/2} N^{-\frac 32 - \frac \rho 4 + 24\eta}\\[2mm]
&\ll N^{7\rho/4 - 1/2} \ll 1.
 \end{align*}
Next, \eqref{eq3.23} holds since
 \[
L^{3/2} Y\, q^{1/2}(YN^{1-24\eta})^{-1} \ll N^{7\rho/4 - 1/2} \ll 1.
 \]

Finally, \eqref{eq3.24} holds since
 \[
Z^{1/2} L^{3/2} Y(YN^{1-24\eta})^{-1} \ll N^{3\rho/2} S^{-1/2} Y^{-1} \ll 1
 \]
from \eqref{eq3.20}.

Now suppose that
 \begin{equation}\label{eq3.25}
Y \ll (N^{2-4\rho-26\rho} S^{-1})^{1/3};
 \end{equation}
we employ \eqref{eq3.16} and \eqref{eq3.19}. We need to verify the bounds
 \begin{align}
Z^{1/2}LY(SLY^2/Z) &\ll YN^{1-25\eta},\label{eq3.26}\\[2mm]
Z^{1/2}LY(SLY^2/q) &\ll YN^{1-25\eta},\label{eq3.27}\\[2mm]
Z^{1/2}LY\, q/Z &\ll YN^{1-25\eta},\label{eq3.28}\\[2mm]
Z^{1/2}LY &\ll YN^{1-25\eta}.\label{eq3.29}
 \end{align}

First of all, \eqref{eq3.26} holds since
 \begin{align*}
L^2 &SY^3Z^{-1/2} (YN^{1-25\eta})^{-1}\\[2mm]
&\le L^4 SY^3 N^{26\eta-2} \ll 1
 \end{align*}
from \eqref{eq3.25}. Next, \eqref{eq3.27} holds since
 \begin{align*}
Z^{1/2} &L^2 SY^3 q^{-1} (YN^{1-25\eta})^{-1}\\[2mm]
&\ll \frac NY\, S^{1/2} L^2 Y^2 N^{-2-\frac \rho 2 + 25\eta}\\[2mm]
&\ll YN^{-1+7\rho/2} \ll 1.
 \end{align*}
Next, \eqref{eq3.28} holds since
 \[
Z^{-1/2}L Yq (YN^{1-25\eta})^{-1} \ll YN^{-1 + 7\rho/2+26\eta} \ll 1.
 \]
Finally, \eqref{eq3.29} holds since
 \[
Z^{1/2} LY(YN^{1-25\eta})^{-1} \ll \frac{N^{25\eta}L}Y \ll 1.
 \]

In order to complete the proof, we show that the ranges of $Y$ in \eqref{eq3.20} and \eqref{eq3.25} overlap. We have
 \[
N^{3\rho/2} S^{-1/2} < (N^{2-4\rho - 26\eta} S^{-1})^{1/3},
 \]
that is
 \[
S^{1/6} > N^{(17\rho - 4 + 52\eta)/6},
 \]
since $\rho < 1/5$. This completes the proof of \eqref{eq3.15}, and \cref{thm3} follows.
 \end{proof}
 \bigskip

\section{Asymptotic formulae via Harman sieve and generalized Vaughan identity}

In the present section and the next, we suppose that $\s = \rho - 1/6$ satisfies
 \begin{equation}\label{eq4.1}
\frac 1{120} < \s \le \frac 1{102}.
 \end{equation}
We write $b = \frac 16 - 5\s$, $f = \frac 13 - 4\s$, $z = N^b$ and
 \[
P(s) = \prod_{p < s} p \quad (s > 1).
 \]
For a finite set $\mc E \subset \mb N$, let
 \begin{align*}
\mc E_d &= \{m : dm \in \mc E\},\\[2mm]
S(\mc E, w) &= |\{m \in \mc E : (m, P(w)) = 1\}|. 
 \end{align*}
As in \cite{bak}, our claim in \cref{thm1} is a corollary of the lower bound
 \begin{equation}\label{eq4.2}
S(\mc A, (2N)^{1/2}) > \frac 1{200}\ \frac{\d N}{\log N}
 \end{equation}
for $\rho = \rho_2$.

We introduce some `comparison' results for the pair $S(\mc A, w)$ and $2\d \, S(\mc B, w)$, and similar pairs, that will be needed in Section 6. First of all, we have
 \smallskip
 
 \begin{equation}\label{eq4.3}
\sum_{\substack{p_1\sim P_1\\
N^\tau \le p_\ell \le \cdots \le p_1\\
p_1 \le p_1' \le \cdots \le p_t'}}
 \cdots \sum_{p_\ell \sim P_\ell} \ \sum_{p_1' \sim Q_1} \cdots \sum_{p_t' \sim Q_t} \Bigg(\sum_{p_1\ldots p_t\, p_1'\ldots p_\ell' \in \mc A} \kern -5pt 1 - 2\d \sum_{p_1\ldots p_t\, p_1' \ldots p_\ell' \in \mc B} \kern -5pt1\Bigg)
 \end{equation}
whenever $\tau$ is a positive constant and some subproduct $R$ of

\noindent $P_1\ldots P_\ell$ $Q_1\ldots Q_t$ satisfies
 \begin{equation}\label{eq4.4}
N^\rho \ll R \ll N^f.
 \end{equation}
This is a consequence of \cref{thm2}; compare the discussion in \cite[Sections 3.2 and 3.5]{har2}. Additional inequalities such as `$p_j \le K$' may be included in the summation in \eqref{eq4.3} without affecting its validity, as explained in \cite[Section 3.2]{har2}.

The following lemma is essentially the special case $M = 2X^\a$, $S=1$ of \cite[Lemma 14]{bakwein}, and is a variant of \cite[Theorem 3.1]{har2}.

 \begin{lem}\label{lem5}
Let $w$ be a complex function with support in $[1,N]$, $|w(n)| \le N^{1/\eta}$ $(n \ge 1)$. Let $0 < \t < \t + \psi < 1/2$. Let
 \[
S(r, v) : = \sum_{(n, P(v)) = 1} w(rn).
 \]
Suppose that, for some $Y > 1$ we have, (for any coefficients $a_m,|a_m| \le 1$, $c_n$, $|c_n| \le 1$ and $b_n$, $|b_n| \le \tau(n))$
 \begin{align}
&\sum_{m \le 2N^\a} a_m \sum_n w(mn) \ll Y,\label{eq4.5}\\[2mm]
&\sum_{N^\t \le h \le N^{\t + \psi}} c_h \sum_n b_n w(mn) \ll Y.\label{eq4.6}
 \end{align}

Let $u_r$ $(r < N^\t)$ be complex numbers with $|u_r| \le 1$, $u_r = 0$ for $(r, P(N^\eta)) > 1$. Then
 \begin{equation}\label{eq4.7}
\sum_{r < N^\t} u_r S(r, N^\psi) \ll Y(\log N)^3.
 \end{equation}
 \end{lem}

We can deduce the following `bilinear' lemma.

 \begin{lem}\label{lem6}
Let $w$, $\t$, $\psi$, $S(r,v)$ be as in \cref{lem5}. Suppose that we have the hypothesis \eqref{eq4.6} and in addition, for some $T \in [1, N)$,
 \begin{equation}\label{eq4.8}
\sum_{m \le 2N^\t} a_m \sum_{t\le T} c_t \sum_n w(mtn) \ll Y
 \end{equation}
for any $a_m$, $c_t$ with $|a_m| \le 1$, $|c_t| \le 1$. Then for any $u_r$ $(r < N^\t)$, $v_t$ $(t \le T)$ with $|u_r| \le 1$, $|c_t| \le 1$, $u_r = 0$ for $(r, P(N^\eta)) > 1$, we have
 \begin{equation}\label{eq4.9}
\sum_{r \le R} u_r \sum_{t\le T} v_t\, S(rt, N^\psi) \ll YN^{2\eta}.
 \end{equation}
 \end{lem}
 
 \begin{proof}
We apply \cref{lem4} with $w$ replaced by $w^*$,
 \[
w^*(n) = \sum_{t \le T} v_t w(nt),
 \]
so that $S(r,v)$ is replaced by
 \[
S^*(r,v) = \sum_{(n, P(v))=1} \ \sum_{t\le T} v_t w(nt).
 \]
From \eqref{eq4.8}, \eqref{eq4.9} the hypotheses of \cref{lem5} are satisfied with $Y$ replaced by $YN^\eta$: for example,
 \begin{align*}
&\sum_{N^\t \le m \le N^{\t + \psi}} a_m \ \sum_n b_n w^*(mn)\\[2mm]
&\qquad = \sum_{N^\t \le m \le N^{\t + \psi}} a_m \sum_{t\le T} \ \sum_n \, b_n v_t w(mtn) \ll YN^\eta
 \end{align*}
(we may group the product $mt$ as a single variable and apply \eqref{eq4.6}). The conclusion \eqref{eq4.7} with $S$ replaced with $S^*$ gives the desired bound \eqref{eq4.9}.
 \end{proof}

We now apply \cref{lem6} with
 \[
w(n) = \chi_{_{\mc A}}(n) - 2\d \chi_{_{\mc B}}(n),
 \]
$\t = \rho$, $\t + \psi = f$, where $T = (2N)^\nu$ and the non-negative number $\nu$ satisfies
 \begin{align}
3\rho + 2\nu &\le \frac 32 - 3\s,\label{eq4.10}\\[2mm]
\rho + 3\nu &\le \frac 74 - \frac{15\s}2,\label{eq4.11}\\[2mm]
\rho + \nu &\le \frac 23 - 8\s.\label{eq4.12}
 \end{align}
 
 \begin{lem}\label{lem7}
Suppose that \eqref{eq4.10}--\eqref{eq4.12} hold. Then
 \begin{equation}\label{eq4.13}
\sum_{r < N^\rho} u_r \ \sum_{t \le (2N)^\nu} v_t(S(\mc A_{rt}, z) - 2\d \, S(\mc B_{rt}, z)) \ll \d N^{1-\eta},
 \end{equation}
whenever $|u_r| \le 1$, $(r, P(N^\eta))=1$ for $u_r \ne 0$, and $|v_t| \le 1$.
 \end{lem}
 
 \begin{proof}
We take $Y = \d N^{1-3\eta}$. The hypothesis \eqref{eq4.8} is a consequence of \cref{thm2} because of \eqref{eq4.10}--\eqref{eq4.12}. The hypothesis \eqref{eq4.6} is a consequence of \cref{thm2}. Now the conclusion \eqref{eq4.9} may be written in the form \eqref{eq4.13}.
 \end{proof}
 
 \begin{lem}\label{lem8}
Let $0 < g \le \frac 16 - 2\s$, $\frac 13 - 4\s < \g < \frac 23 - 8\s$. Let $\rho_1 \ge \cdots \ge \rho_t \ge 0$ with $\rho_1 + \cdots + \rho_t = \g$,  $\rho_1 \le \g - g$. There is a set $\mc C \subset \{1, \ldots, t\}$ with $\sum\limits_{i\in \mc C} \rho_i \in \left[g, \frac 13 - 4\s\right]$. 
 \end{lem}
 
 \begin{proof}
Suppose that no such $\mc C$ exists. Now suppose first that $\rho_1 \le g$. Since $2g \le \frac 13 - 4\s$ we can prove successively that $\rho_1 + \rho_2, \ldots, \rho_1 + \cdots + \rho_t$ are in $[0,g]$. This is absurd. 

Thus we must have $\rho_1 > \frac 13 - 4\s$. But now $\rho_2 + \cdots + \rho_t = \g - \rho_1 < \g - \left(\frac 13 - 4\s\right) < \frac 13 - 4\s$, and $\rho_2 + \cdots + \rho_t \ge g$ since $\rho_1 \le \g - g$. This is absurd.
 \end{proof}
 
 \begin{lem}\label{lem9}
Let $F$ be a complex function on $[1,N]$. The sum\newline $\sum\limits_{k\le N} \Lambda(k)$ $F(k)$ may be decomposed into at most $C(\log N)^8$ sums of the form
 \[
\sum_{n_i\in I_i;\ n_1\ldots n_8 \le N} (\log n_1)\mu(n_5) \ldots \mu(n_8)\ \ F(n_1\ldots n_8)
 \]
where $I_i = (N_i, 2N_i]$, $\prod N_i < N$ and $2N_i \le N^{1/4}$ if $i > 4$.
 \end{lem}
 
 \begin{proof}
This is a case of Heath-Brown's `generalized Vaughan identity' \cite{hb}.
 \end{proof}
 
 \begin{lem}\label{lem10}
Let
 \begin{equation}\label{eq4.14}
(2N)^{3/8 + 33\s/4} \le Q < Q' \le (2N)^{\frac 12}, \ Q' \le 2Q. 
 \end{equation}
We have 
 \begin{equation}\label{eq4.15}
\sum_{Q \le p < Q'} (S(\mc A_p, z) - 2\d\, S(\mc B_p, z)) \ll \d N^{1-\eta}.
 \end{equation}
 \end{lem}
 
 \begin{proof}
Arguing as in the proof of \eqref{eq4.3}, it will suffice to show that
 \[
\sum_{Q \le p < Q'} \ \sum_{\ell \le L} c_\ell \ \sum_{\substack{\frac N2 < n \le 2N\\
(n, P(z)) = 1}} e(\ell g(pn)) \ll N^{1-2\eta}
 \]
for $|c_\ell| \le 1$. By a partial summation argument, it suffices to obtain
 \[
\sum_{Q \le m < Q'} \Lambda (m) \sum_{\ell \le L} c_\ell \ \sum_{\substack{\frac N2 < n \le 2N\\
(n,P(z)) =1}} e(\ell g(mn)) \ll N^{1-2\eta}.
 \]
Applying \cref{lem9}, we need only show that
 
 \begin{equation}\label{eq4.16}
\sum_{\substack{Q \le n_1\ldots n_8 < Q'\\
n_i \sim N_i\ \forall i}} (\log n_1)\mu(n_5)\ldots \mu(n_8) \sum_{\ell \le L} c_\ell \ \sum_{\substack{\frac N2 < n \le 2N\\
(n, P(z)) =1}} e(\ell g(n_1\ldots n_8n)) \ll N^{1-3\eta}
 \end{equation}
whenever $\prod\limits_i N_i < N$ and $2N_i \le N^{1/4}$ for $i > 4$. Thus $Q \ll N_1 \ldots N_8 \ll Q$. We write $N_1\ldots N_8 = N^\g$.

We may assume in view of \cref{thm3} that no subproduct $X$ of $N_1, \ldots, N_8$ satisfies
 \begin{equation}\label{eq4.17}
N^\rho \ll X \ll N^f. 
 \end{equation}
We now reorder $N_1, \ldots, N_k$ as $(N_1\ldots N_k)^{\rho_j}$ $(1 \le j \le k)$ with $\rho_1 \ge \cdots \ge \rho_k \ge 0$. Let $g = \min\left(3\g - \frac 54 - \frac{15\s}2, \g - \frac 13 - 8\s\right)$, so that $0 < g \le \g - \frac 13 - 8\s$ and
 \[
g  = \begin{cases}
3\g - \dfrac 54 - \dfrac{15\s}2 & \left(\g \le \dfrac{11}{24} - \dfrac \s 4\right)\\[4mm]
\g - \dfrac 13 - 8\s & \left(\g > \dfrac{11}{24} - \dfrac \s 4\right).
\end{cases} 
 \]
We divide the argument into two cases.
 \begin{case}
We have $\rho_1 \ge \g - g > \frac 14$. Thus $i \le 4$ and (after a partial summation if necessary) we can apply \cref{thm2} to the sum in \eqref{eq4.15} with
 \[
V \ll N^{\g - \rho_1}\, , \ W \ll N^{1-\g}. 
 \]
We verify the hypotheses of \cref{thm2}. First,
 \[
3(\g - \rho_1) + 2(1 - \g) \le 3g + 2- 2\g \le \frac 32 - 3\s = 2 - 3\rho, 
 \]
since
 \[
g \le \g - \frac 13 - 8\s \le \frac 13 \left(2\g - 3\s - \frac 12\right). 
 \]
Next,
 \[
(\g - \rho_1) + 3(1 - \g) \le g + 3 - 3\g \le \frac 74 - \frac{15\s}2 = 3 - \frac{15\rho}2
 \]
and
 \[
(\g - \rho_1) + (1 - \g) \le g + 1 - \g \le \frac 23 - 8\s = 2 - 8\rho 
 \]
from the definition of $g$. Now \eqref{eq4.15} follows from \cref{thm2}.
 \end{case}
 
 \begin{case}
We have $\rho_1 < \g - g$. By \cref{lem8} and the absence of a product $X$ satisfying \eqref{eq4.17}, there is a subsum $u = \sum\limits_{i \in \mc C} \rho_i$ such that
 \[
g \le u < \frac 16 + \s. 
 \]
We are now in a position to apply \cref{lem6} with $\t = \rho$, $\t + \psi = f$,
 \begin{align*}
w(n) &= \begin{cases}
 \ds \sum_{\ell \le L} c_\ell e(\ell g(n)) & \left(\dfrac N2 < n \le N\right)\\[6mm]
 0 & \text{(otherwise)}, 
 \end{cases}\\
\intertext{and}
 T &\ll N^{\g - u}.
 \end{align*}
We need to verify \eqref{eq4.6}, \eqref{eq4.8}. Clearly \eqref{eq4.8} is a consequence of \cref{thm3}. As for \eqref{eq4.8}, we need to verify the hypotheses of \cref{thm2} with $V \le 2N^{1/6 +\s}$, $W\ll N^{\g-u}$. Suppose first that $\g > \frac{11}{24} - \frac \s 4$. Then $\g - u \le \frac 13 + 8\s$,
 \begin{align*}
3\left(\frac 16 + \s\right) &+ 2\left(\frac 13 + 8\s\right) = \frac 76 + 19\s < \frac 32 - 3\s;\\[2mm]
\left(\frac 16 + \s\right) &+ 3\left(\frac 13 + 8\s\right) = \frac 76 + 25\s < \frac 74 - \frac{15\s}2;\\[2mm]
\left(\frac 16 + \s\right) &+ \left(\frac 13 + 8\s\right) = \frac 12 + 9\s \le \frac 23 - 8\s.
 \end{align*}
Now suppose that $\g \le \frac{11}{24} - \frac \s 4$. Then $\g - u \le \frac 54 + \frac{15\s}2 - 2\g$,
 \begin{align*}
3\left(\frac 16 + \s\right) &+ 2\left(\frac 54 + \frac{15\s}2 - 2\g\right) = 3 + 18\s - 4\g\\[2mm]
&\hskip .75in \le 3 + 18\s - 4\left(\frac 38 + \frac{33\s}4\right) < \frac 32 - 3\s,
 \end{align*}
 \newpage
 
 \begin{align*}
\left(\frac 16 + \s\right) &+ 3\left(\frac 54 + \frac{15\s}2 - 2\g\right) = \frac{47}{12} + \frac{47\s}2 - 6\g\\[2mm]
&\hskip .75in \le \frac{47}{12} + \frac{47\s}2 - 6\left(\frac 38 + \frac{33\s}4\right) < \frac 74 - \frac{15\s}2
 \end{align*}
and
 \[
\left(\frac 16 + \s\right) + \left(\frac 54 + \frac{15\s}2 - 2\g\right) = \frac{17}{12} + \frac{17\s}2 - 2\g \le \frac 23 - 8\s.
 \]
Thus \cref{thm2} yields the desired estimate \eqref{eq4.8}. Now the lemma follows from \cref{lem6}.\qedhere
 \end{case}
 \end{proof}

There is a short interval in which we can use \cref{lem9} directly to obtain a conclusion stronger than \eqref{eq4.15}.
 
 \begin{lem}\label{lem11}
We have
 \[
\sum_{Q \le p \le Q'} (S(\mc A_p, p) - 2\d S(\mc B_p, p)) \ll \d N^{1-\eta} 
 \]
whenever
 \[
(2N)^{1/3 + 8\s} \le Q < Q' \le (2N)^{\frac 12 - 9\s}, Q' \le 2Q. 
 \] 
 \end{lem}
 
 \begin{proof}
In this range of $Q$ we have
 \[
\sum_{Q \le p < Q'} S(\mc A_p, p) = \sum_{\substack{Q \le p < Q'\\
pp' \in \mc A}} 1; 
 \]
similarly with $\mc B$ in place of $\mc A$. We apply \cref{lem9} to decompose the sum over $p'$ (with $\Lambda(m)$ in place of $p'$). Clearly it will be enough to show that, for $|c_\ell| \le 1$,

 \begin{equation}\label{eq4.18}
\sum_{Q \le p < Q'} \ \sum_{\substack{\frac N2 < pn_1 \ldots n_8 \le N\\
n_i\sim N_i}} (\log n_1)\mu(n_5)\ldots \mu(n_8) \sum_{\ell \le L} c_\ell e (\ell g(pn_1 \ldots n_8)) \ll N^{1-3\eta},
 \end{equation}
where $N \ll Q \prod\limits_{i=1}^8 N_i \ll N$ and $N_i \ll (N/Q)^{1/4}$ for $i \ge 5$. As in the preceding proof we may suppose that no subproduct $X$ of $N_1 \ldots N_8$ satisfies \eqref{eq4.17}. Writing $Q = N^\g$, we have
 \[
N^{\frac{1-\g}2} \ll N^{\frac 13 - 4\s}.
 \]
Thus it is clear by a `reflection' argument that no such $X$ can satisfy
 \[
N^{\frac 16 + \s} \ll X \ll N^{1 - \g - \left(\frac 16 - \s\right)} = N^{\frac 56 - \g - \s}.
 \]
Moreover,
 \[
\frac 56 - \g - \s > \frac 56 - \left(\frac 12 - 9\s\right) - \s > 2\left(\frac 16 + \s\right).
 \]
It follows that
 \[
\prod_{N_i \le N^{\frac 16 +\s}} N_i \ \le N^{\frac 16 + \s}.
 \]
There cannot be two indices $i$ with $N_i > N^{\frac 56 - \g - \s}$, since
 \[
\frac 53 - 2\g - 2\s > 1 - \g.
 \]
Hence there is a $j$ with
 \[
N_j > N^{\frac 56 - \g - \s}\ , \ N^{1-\g}/N_j < N^{\frac 16 + \s}.
 \]
We are now in a position to apply \cref{thm2} with
 \[
V \ll N^{\frac 16 + \s} \ , \ W \ll N^\g.
 \]
We make the usual verification:
 \begin{align*}
3\left(\frac 16 + \s\right) + 2\g &\le 3\left(\frac 16 + \s\right) + 2\left(\frac 12 - 9\s\right) < \frac 32 - 3\s,\\[2mm]
\left(\frac 16 + \s\right) + 3\g &\le \frac 16 + \s + 3\left(\frac 12 - 9\s\right) < \frac 74 - \frac{15\s}2;\\[2mm]
\left(\frac 16 + \s\right) + \g &\le \frac 16 + \s + \frac 12 - 9\s = \frac 23 - 8\s. 
 \end{align*}
Thus \cref{thm2} yields \eqref{eq4.18}. This completes the proof of \cref{lem11}.
 \end{proof}
 \bigskip
 
 \section{Application of the linear sieve}\label{sec5}

In order to obtain an upper bound for
 \begin{equation}\label{eq5.1}
\sum_{P \le p < P'} S(\mc A_p, p) 
 \end{equation}
whenever
 \begin{equation}\label{eq5.2}
(2N)^{\frac 12 - 9\s} \le P \le P' \le (2N)^{\frac 38 + \frac{33\s}4}, \ P' \le 2P, 
 \end{equation}
we apply Theorem 4 of Iwaniec \cite{iwan}, which we state in a form sufficient for our needs. The quantity estimated will actually exceed that in \eqref{eq5.1}, which will be exploited in Section 6. 

Let $\mc E$ be a set of integers in $[1, N]$. Fix an approximation $X$ to $|\mc E|$ and write
 \[
r(\mc E, d) = |\mc E_d| - \frac Xd.
 \]
Let $F(s)$ be the upper bound function for the linear sieve \cite[p. 309]{iwan}. In the following lemma, let $D \ge Z \ge 2$. Let
 \begin{align*}
\mc G &= \{D^{\e^2(1+\eta)^n} : n \ge 0\},\\[2mm]
\mc H &= \{\bs D = (D_1, \ldots, D_r) : r \ge 1 \ , \ D_\ell \in \mc G \ \text{for } 1 \le \ell \le r,\\[2mm]
&\hskip 1.75in D_r \le \cdots \le D_1 < D^{1/2}\},\\
\mc D^+ &= \left\{\bs D \in \mc H : D_1 \ldots D_{2\ell} D_{2\ell+1}^3 < D \ \text{for } 0 \le \ell \le \frac{r-1}2\right\}.
 \end{align*}
 
 \begin{lem}\label{lem12}
With the above notations, we have 
 \begin{align*}
S(\mc E, Z) &\le V(Z) X\left\{F\left(\frac{\log D}{\log Z}\right) + E\right\} + R_1^+ + R^+,\\
\intertext{where}
E &< C(\e + \e^{-8}(\log D)^{-1/3}),\\[2mm]
R_1^+ &= \sum_{\substack{d < D^\e\\
d\mid P(D^{\e^2})}} \phi_d^+(D^\e) r(\mc A, d),\\[2mm]
R^+ &= \sum_{\bs D \in \mc D^+} \ \sum_{\substack{d < D^\e\\
d\mid P(D^{\e^2})}} \Lambda_d^+(\e, D) H_d(\mc E, Z, \e, \bs D),
 \end{align*}
with some coefficients $\phi_d^+(D^\e)$ and $\Lambda_d^+(\e, \bs D)$ bounded by $1$ in absolute value. Here
 \[
H_d(\mc E, Z, \e, \bs D) = \sum_{D_i \le p_i < D_i^{1+\eta}, \, p_i < Z \ (1 \le i \le r)} r(\mc E, dp_1\ldots p_r).
 \]
 \end{lem}

In our application, we shall take
 \[
\mc E = \mc A_p
 \]
for any $p \in [P, P')$, and
 \[
X = 2\d|\mc B_p|.
 \]
We write $P = N^\g$. We shall take $D = N^{\frac 23 - 8\s - \g - 2\e}$ and $Z = N^{\frac 1{12} + \frac \s 2}$. It is easily verified that $D^{\frac 13} \le Z \le D^{1/2}$, so that
 \begin{equation}\label{eq5.3}
V(Z) F\left(\frac{\log D}{\log Z}\right) = \frac 2{\log D}\, (1 + \l \e). 
 \end{equation}
We apply \cref{lem11} and sum over $p$, obtaining (on noting that $P > D^\d Z$)
 \begin{align}\label{eq5.4}
\sum_{P \le p < P'} &S(\mc A_p, Z)\\[2mm]
&\le V(Z) \left\{F\left(\frac{\log D}{\log Z}\right) + C\e\right\} \sum_{P \le p < P'} 2\d |\mc B_p| + E_1 + E_2.\notag
 \end{align}
Here
 \begin{align*}
E_1 &= \sum_{\substack{d < D^\e\\
d\mid P(D^{\e^2})}} \ \sum_{P \le p < P'} r(\mc A, pd),\\[2mm]
E_2 &= \sum_{\substack{d < D^\e\\
d\mid P(D^{\e^2})}} \ \sum_{P \le p < P'} \sum_{\bs D \in \mc D^+} \Lambda_d^+ (\e, \bs D) \ \sum_{\substack{
D_i \le p_i < D_i^{1+\eta}\, (1 \le i \le r)\\
p_i < Z \ (1 \le i \le r)}} r(\mc A, pdp_1 \ldots p_r).
 \end{align*}

We shall show that
 \begin{equation}\label{eq5.5}
E_2 \ll \d N^{1-\eta};
 \end{equation}
the proof that $E_1 \ll \d N^{1-\eta}$ is similar but simpler.

Reducing the task of bounding $Z$ to estimating exponential sums as in previous sections, it suffices to prove that for $|c_\ell| \le 1$, and a fixed $\bs D \in \mc H$, we have
 \begin{align}
\sum_{\ell \le L} c_\ell \sum_{\substack{d < D^\e\\
d\mid P(D^{\e^2})}}& \Lambda_d^+(\e,\bs D) \sum_{
P \le p < P'} \ \sum_{\substack{
D_i\le p_i < \min(D_i^{1+\eta}, Z)\\
(1 \le i \le r)}}\label{eq5.6}\\[2mm]
&\sum_{\frac N2 < dp_1\ldots p_r n \le N} e(\ell g(dpp_1\ldots p_rn)) \ll N^{1-3\eta}.\notag
 \end{align}
This is a consequence of \cref{thm2}, grouping the variables as $v = dp_1\ldots p_r$ and $p$,
 \[
v \ll N^\e D^{1+\eta} \ll N^{\frac 23 - 8\s - \g}, p \le 2N^\g.
 \]
We make the usual verifications:
 \begin{align*}
3\left(\frac 23 - 8\s - \g\right) &+ 2\g = 2 - 24\s - \g < \frac 32 - 3\s,\\[2mm]
\left(\frac 23 - 8\s - \g\right) &+ 3\g = \frac 23 - 8\s + 2\g\\[2mm]
&\le \frac 23 - 8\s + \frac 34 + \frac{33\s}2 < \frac 74 - \frac{15\s}2,\\[2mm]
\left(\frac 23 - 8\s - \g\right) &+ \g = \frac 23 - 8\s.
 \end{align*}
The desired bound \eqref{eq5.5}, and the same bound for $E_2$, now follows.  Now \eqref{eq5.3}, \eqref{eq5.4} yield
 \begin{equation}\label{eq5.7}
\sum_{P \le p < P'} S(\mc A_p, Z) \le \frac{2(1+C\e)}{\frac 23 - 8\s - \g} \, \d N \, \sum_{P \le p < P'} \frac 1p.
 \end{equation}
  \bigskip
 
 \section{The sieve decomposition}\label{sec6}

For the final stage of our work we take $\rho = \frac{37}{210}+\e$ so that $\s = \frac 1{105} + \e$. In this section, each $S_j$ $(j \ge 0)$ that occurs takes the form
 \[
S_j = \sum_{\substack{1 \le r \le 8\\
1 \le t \le 8}} \ \ \sideset{}{^*}\sum_{p_1\ldots p_rp_1'\ldots p_t' \in \mc A} 1
 \]
 \newpage
\noindent where the asterisk indicates a restriction of $p_1\ldots p_r$ and $p_1' \ldots p_t'$ to certain subsets of $[1,N]$ depending on $j$; $S_j'$ is obtained from $S_j$ on replacing $\mc A$ by $\mc B$. For some of these values of $j$ we write
 \begin{equation}\label{eq6.1}
S_j = K_j + D_j
 \end{equation}
where $K_j$ is defined by the following additional condition of summation within $S_j$: \textit{a subproduct $R$ of $p_1\ldots p_r$ $p_1'\ldots p_t'$ satisfies}
 \begin{equation}\label{eq6.2}
N^\rho \ll R \ll N^f.
 \end{equation}
We split up $S_j'$ as $S_j' = K_j' + D_j'$ in the same way. As noted in Section 5,
 \begin{align}
S_j \ge K_j &\ge 2\d K_j' \ (1 - C\e)\label{eq6.3}\\[2mm]
&\ge 2\d K_j' - \frac{C\d \e N}{\log N}\notag\\[2mm]
&\ge 2\d S_j' - 2\d D_j' - \frac{C\d \e N}{\log N}\notag
 \end{align}
whenever \eqref{eq6.1} is used.

Conversion of sums into integrals, with an acceptable error, in the following is along familiar lines (see \cite{har2}). Concerning Buchstab's function $\omega(t)$, we note that
 \begin{equation}\label{eq6.4}
\omega(t) \le \k \text{ for } t \ge \frac 1\k
 \end{equation}
provided that $\k \ge 0.5672$, and that
 \begin{equation}\label{eq6.5}
\omega(t) \ge e^{-\g} - 2.1 \times 10^{-8} \quad (t \ge 6);
 \end{equation}
see Cheer and Goldston \cite{cheergold}.

Let $S_0 = S(\mc A, (2N)^{1/2})$. Using Buchstab's identity and writing $p_i = (2N)^{\a_i}$, we have
 \begin{equation}\label{eq6.6}
S_0 = S_1 - \sum_{j=2}^7 S_j
 \end{equation}
where

 \begin{gather*}
S_1 = S(\mc A, z),\\[2mm]
S_j = \sum_{\a_1 \in I_j} S(\mc A_{p_1}, p_1) \quad (2 \le j \le 7),
 \end{gather*}
with $I_2 = \left[\frac 16 - 5\s, \frac 16 + \s\right)$,$I_3 = \left(\frac 16 + \s, \frac 13 - 4\s\right)$, $I_4 = \left[\frac 13 - 4\s, \frac 13 + 8\s\right)$, $I_5 = \left[\frac 13 + 8\s, \frac 12 - 9\s\right)$, $I_6 = \left[\frac 12 - 9\s, \frac 38 + \frac{33\s}4\right)$ and $I_7 = \left[\frac 38 + \frac{33\s}4, \frac 12\right]$.
 \medskip

Recalling \eqref{eq4.3} and Lemmas \ref{lem7} (with $\nu = 0$) and \ref{lem11}, we have
 \begin{equation}\label{eq6.7}
S_j = 2\d S_j' (1 + \l \e) 
 \end{equation}
for $j=1, 3, 5$. For $j=2$, we apply Buchstab's identity three further times to obtain
 \begin{equation}\label{eq6.8}
S_2 = S_8 - S_9 + S_{10} - S_{11}
 \end{equation}
where
 \begin{gather*}
S_8 = \sum_{b \le \a_1 < \rho} S(\mc A_{p_1}, z), \ S_9 = \sum_{b \le \a_2 \le \a_1 < \rho} S(\mc A_{p_1p_2}, z),\\[2mm]
S_{10} = \sum_{b \le \a_3 \le \a_2 \le \a_1 < \rho} S(\mc A_{p_1p_2p_3}, z)\\
\intertext{and}
S_{11} = \sum_{b \le \a_4 \le \a_3 \le \a_2 \le \a_1 < \rho} S(\mc A_{p_1p_2p_3p_4}, p_4). 
 \end{gather*}

We can apply \cref{lem7} to $S_8$, $S_9$, $S_{10}$ to obtain \eqref{eq6.7}: for example, in $S_{10}$,
 \begin{align*}
3\rho + 2(\a_1 + \a_2) &\le 7\rho \le \frac 32 - 2\s\\[2mm]
\rho + 3(\a_1 + \a_2) &\le 7\rho \le \frac 74 - \frac{15\s}2,\\[2mm]
\rho + (\a_1 + \a_2) &\le 3\rho \le \frac 23 - 8\s.
 \end{align*}
We treat $S_{11}$ via \eqref{eq6.3}.

For $j=4$ we apply Buchstab's identity once. Iterating once more for part of the sum over $p_1$, $p_2$, this gives
 \begin{equation}\label{eq6.9}
S_4 = S_{12} - S_{13} - S_{14} + S_{15}, 
 \end{equation}
where
 \begin{align*}
S_{12} &= \sum_{\a_1\in I_4} S(\mc A_{p_1},z), \ S_{13} = \sum_{\substack{\a_1 \in I_4\\
\a_2\in \left[\frac f2, \a_1\right)}} S(\mc A_{p_1p_2},p_2),\\[2mm]
S_{14} &= \sum_{\substack{\a_1 \in I_4\\
\a_2 \in \left[b, \frac f2\right)}} S(\mc A_{p_1p_2}, z), \ S_{15} = \sum_{\substack{\a_1\in I_4\\
b \le \a_3 \le \a_2 < f/2}} S(\mc A_{p_1p_2p_3}, p_3). 
 \end{align*}

We have \eqref{eq6.7} for $S_{12}$, $S_{14}$ since \cref{lem9} is applicable. For example, for $S_{14}$ we have
 \begin{align*}
3\rho + 2\a_2 &\le \frac 12 + 3\s + \frac 23 + 16\s < \frac 32 - 3\s,\\[2mm]
\rho + 3\a_2 &\le \frac 16 + \s + 1 + 24\s < \frac 74 - \frac{15\s}2,\\[2mm]
\rho + \a_2 &\le \frac 16 + \s + \frac 13 + 8\s < \frac 23 - 8\s. 
 \end{align*}
We have \eqref{eq6.7} also for $S_{15}$, this time using \eqref{eq4.3}, since $\rho \le 2b \le \a_2 + \a_3 < f$ in $S_{15}$. For $S_{13}$, we use the lower bound \eqref{eq6.3}.

We also apply Buchstab once more to $S_7$,
 \begin{equation}\label{eq6.10}
S_7 = S_{16} - S_{17} 
 \end{equation}
where
 \[
S_{16} = \sum_{\a_1 \in I_7} S(\mc A_{p_1}, z)
 \]
satisfies \eqref{eq6.7} by \cref{lem10}, and
 \[
S_{17} = \sum_{\substack{
\a_1 \in I_7\\
b\le \a_2 < \a_1}} S(\mc A_{p_1p_2},p_2)
 \]
is bounded below as in \eqref{eq6.3}.

For $S_6$, we proceed differently. We have
 \begin{equation}\label{eq6.11}
S_6 + S_{18} = \sum_{\a_1\in I_6} S(\mc A_{p_1}, Z)
 \end{equation}
where
 \[
S_{18} = \left|\left\{p_1 p_1' \ldots p_r' \in \left(\frac N2, N\right]: \a_1 \in I_6, r \ge 2, Z \le p_1' \le \cdots \le p_r'\right\}\right| 
 \]
is treated as in \eqref{eq6.3}.

By \eqref{eq5.7},
 \begin{equation}\label{eq6.12}
\sum_{\a_1\in I_6} S(\mc A_{p_1}, Z) \le \frac{\d N(1 + C\e)}{\log N} \int_{I_7} \frac{d\a_1}{\a_1} \ \frac 2{\frac 23 - 8\s - \a_1}. 
 \end{equation}
Moreover,
 \begin{equation}\label{eq6.13}
S_6' + S_{18}' = \frac{\d N}{\log N}\, (1 + \l \e) \int_{I_7} \frac{d\a_1}{\a_1}\ \frac 2b\, \omega \left(\frac{1-\a_1}{b/2}\right).
 \end{equation}

Combining \eqref{eq6.11}--\eqref{eq6.13} and \eqref{eq6.3} with $j=18$,
 \begin{align}
S_6 \le 2\d S_6' + \frac{\d N}{\log N} &\int_{I_7} \frac{d\a_1}{\a_1}\, \left(\frac 2{\frac 23 - 8\s - \a_1} - \frac 2b\, \omega \left(\frac{1-\a_1}{b/2}\right)\right)\label{eq6.14}\\[2mm]
&+ 2\d D_{18}' + \frac{C\e \d N}{\log N}.\notag
 \end{align}

Our sieve decomposition, obtained by combining \eqref{eq6.6}, \eqref{eq6.8}, \eqref{eq6.9} and \eqref{eq6.10}, is
 \begin{align*}
S_0 &= S_1 - S_3 - S_5 - (S_8 - S_9 + S_{10} - S_{11})\\
&- (S_{12} - S_{13} - S_{14} + S_{15}) - (S_{16} - S_{17}) - S_6 
 \end{align*}
and also holds if $S_j$ is replaced by $S_j'$. Combining all applications of \eqref{eq6.7} and \eqref{eq6.3} with \eqref{eq6.14}, we end up with
 \begin{align}
S_0 \ge S_0' &- 2\d(D_{11}' + D_{13}' + D_{17}' + D_{18}')\label{eq6.15}\\[2mm]
&- \frac{\d N}{\log N} \int_{I_7} \frac{d\a_1}{\a_1} \left(\frac 2{\frac 23 - 8\s - \a_1} - \frac 2b\, \omega \left(\frac{1-\a_1}{b/2}\right)\right)\notag\\[2mm]
&\hskip 1.75in - \frac{C\e \d N}{\log N}.\notag
 \end{align}

Our next task is to evaluate $D_{11}'$, $D_{13}'$, $D_{17}'$, and $D_{18}'$ with sufficient accuracy. Using $2b > c$ we find that in $D_{11}$, $\a_3 + \a_4 \notin [c,f]$ implies $\a_3 + \a_4 > f$. With a little thought we find that
 \begin{align}
D_{11}' &\le \sum_{\substack{b \le \a_4 \le \a_3 \le \a_2 \le \a_1 < \rho,\, \a_3 + \a_4 > f\\
p_1p_2p_3p_4n_4 \in \left(\frac N2, N\right], \, p\mid n_4 \to p\ge p_4}}1\label{eq6.16}\\[2mm]
&\le (1 + \e)\, \frac N{2\log N}\, J_1,\notag\\
\intertext{where}
J_1 &= \iiiint\limits_{\substack{b \le w \le z \le y \le x < \rho\\
z + w > f}} \frac{dx}x \ \frac{dy}y \ \frac{dz}z \ \frac{dw}{w^2} \, \omega \left(\frac{1 - x - y - z - w}w\right).\notag
 \end{align}
It is simplest to replace the $\omega$ factor by $0.59775$ using \eqref{eq6.4}. Carrying out the $x$ and $w$ integrations (and using $z > f/2$) leads to
 \begin{equation}\label{eq6.17}
J_1 \le 0.59775 \int_{\frac f2 < z \le y \le \rho} \frac{dy}y \ \frac{dz}z \left(\frac 1{f-z} - \frac 1z\right)\log\, \frac \rho y < 0.000691. 
 \end{equation}

For $D_{13}'$, $D_{17}'$ we write the numbers $p_1,p_2, p_1', \ldots, p_t'$ that appear in $D_j'$ as
 \[
(2N)^{\a_1}, (2N)^{\a_2}, (2N)^{\b_1}, \ldots, (2N)^{\b_k}, (2N)^{\g_1}, \ldots, (2N)^{\g_\ell}
 \]
where $k \ge 0$, $\ell \ge 0$, $k + \ell \ge 1$, 
 \begin{gather*}
\a_2 \le \b_1 \le \cdots \le \b_k < \rho,\\[2mm]
f < \g_1 \le \cdots \le \g_\ell\\
\intertext{and}
\a_1 + \a_2 + \b_1 + \cdots + \b_k + \g_1 + \cdots + \g_\ell = 1 + \l \e.
 \end{gather*}

We now consider $D_{13}'$. First we treat the contribution $D^{(1)}$ from $\a_2 > f$. Thus $k=0$. We cannot have $\ell \ge 2$, since
 \[
4\left(\frac 13 - 4\s\right) > 1. 
 \]
If we consider $\ell=1$, we have $\a_1 + 2\a_2 < 1$ and obtain
 \begin{gather}
D^{(1)} = (1 + \l \e) \ \frac{\d N}{\log N} \ J_2,\label{eq6.18}\\
\intertext{where}
J_2 = \int_{\frac 13 - 4\s}^{\frac 13 + 8\s} \int_{\frac 13 - 4\s}^{\min\left(x, \frac{1-x}2\right)} \frac{dx\, dy}{xy(1 - x- y)} < 0.059343.\label{eq6.19}
 \end{gather}

Let $D^{(2)}$ be the contribution to $D_{13}$ from $\a_2 < \rho$. We have $\a_2 + \b_1 \ge f$, $\b_1 \ge f/2$, by an argument used above. We cannot have $\ell \ge 2$ since
 \begin{equation}\label{eq6.20}
3\left(\frac 13 - 4\s\right) + \frac 16 - 2\s = \frac 76 - 14\s > 1.
 \end{equation}
If $\ell = 1$, we must have $k \le 1$ (use \eqref{eq6.20} again). If $\ell = 0$, we must have $k \le 3$ similarly; however, $k>2$, since
 \[
\frac 13 + 8\s + 3\left(\frac 16 + \s\right) < 1.
 \]
The three remaining cases lead to
 \begin{equation}\label{eq6.21}
D^{(2)} = (1 + \l \e) (J_3 + J_4 + J_5),
 \end{equation}
where
 \begin{align}
J_3 &= \int_{\frac 13 - 4\s}^{\frac 13 + 8\s} \int_{\frac 16 - 2\s}^{\frac 16} \frac{dy dx}{xy(1-x-y)} < 0.118914 \quad (\ell = 1, k = 0),\label{eq6.22}\\[4mm]
J_4 &= \int_{\frac 13 - 4\s}^{\frac 13 + 8\s} \int_{\frac 16 - 2\s}^{\frac 16 + \s} \int_{\max\left(y, \frac 13 - 4\s - y\right)}^{\frac 16 + \s} \frac{dzdydx}{xyz(1-x-y-z)} < 0.027404 \quad (\ell = 1, k = 1),\label{eq6.23}\\[4mm]
J_5 &= \int_{\frac 13 - 4\s}^{\frac 13 + 8\s} \int_{\frac 16 - 2\s}^{\frac 16 + \s} \int_{\max\left(y, \frac 13 - 4\s - y\right)}^{\frac 16 + \s} \int_{\frac 56 - \s - (x+y+z)}^{(1 - x - y - z)/2} \frac{dwdzdydx}{xyzw(1-x-y-z-w)}\label{eq6.24}\\[2mm]
&< 0.000245 \quad (\ell = 0, k = 3).\notag
 \end{align}
(The integral was bounded above by an integral in $x$, $y$, $z$, using $1-x-y-z-w \ge w$, before carrying out the $w$ integration.)

In $D_{17}'$, we cannot have $\a_2 \ge \rho$, since then
 \[
\a_2 < \frac 12\, (1 - \a_1) < \frac 12 \left(\frac 58 - \frac{33\s}4\right) < \frac 13 - 4\s, 
 \]
which is absurd. Thus $\a_2 < \rho$; we cannot have $\ell \ge 2$ or $\ell = 1$, $k \ge 1$ or $\ell = 0$, $k \ge 3$ since
 \[
\frac 38 + \frac{33\s}4 + 2\left(\frac 13 - 4\s\right) > 1.
 \]
For $\ell = 0$, we have $k > 1$ since
 \[
\frac 12 + 2\left(\frac 16 + \s\right) < 1.
 \]
Thus
 \begin{equation}\label{eq6.25}
D_{17}' = (1 + \l \e) \, \frac N{2\log N}\, (J_6 + J_7),
 \end{equation}
where
 \begin{align*}
J_6 &= \int_{\frac 38 + \frac{33\s}4}^{\frac 12} \int_{\frac 16 - 5\s}^{\frac 16 + \s}\ \frac{dydx}{xy(1-x-y)} < 0.060205 \ (\ell = 1, k = 0);\\[4mm]
J_7 &= \underset{\frac 16 + \s > 1 - x - y - z > z}{\int_{\frac 38 + \frac{33\s}4}^{\frac 12} \int_{\frac 16 - 2\s}^{\frac 16 + \s} \int_y^{\frac 16 + \s}}\ \frac{dzdydx}{xyz(1-x-y-z)} < 0.000237\\[2mm]
&\hskip 3in (\ell = 0, k = 2).
 \end{align*}
 
In $D_{18}'$, we write the numbers that appear as $p_1, p_1', \ldots, p_r'$ as
 \[
(2N)^{\a_1}, (2N)^{\b_1} \le \cdots \le (2N)^{\b_k}, (2N)^{\g_1} \le \cdots \le (2N)^{\g_k}
 \]
where $k + \ell = r \ge 2$, $k \ge 0$, $\ell \ge 0$, $\b_k < \rho$, $\g_1 > f$,
 \[
\a_1 + \b_1 + \cdots + \b_k + \g_1 + \cdots + \g_k = 1 + \l_9\e.
 \]
We observe that
 \[
\frac 12 - 9\s + 2\left(\frac 13 - 4\s\right) > 1.
 \]
Thus we cannot have $\ell \ge 2$ or $\ell = 1$, $k \ge 2$ or $\ell = 0$, $k \ge 4$. We cannot have $\ell = 0$, $k \le 3$ since
 \[
\frac 38 + \frac{33\s}4 + 3\left(\frac 16 + \s\right) < 1.
 \]
Since $\ell + k \ge 2$, the only remaining possibility is $\ell = 1$, $k=1$, so that
 \begin{equation}\label{eq6.26}
D_{18}' = (1 + \l \e)\, \frac N{2\log N} \, J_8 
 \end{equation}
where
 \[
J_8 = \int_{\frac 12 - 9\s}^{\frac 38 + \frac{33\s}8} \int_{\frac 1{12} + \frac \d 2}^{\frac 16 + \s} \ \frac{dy dx}{xy(1-x-y)}.
 \]
Combining this with \eqref{eq6.14} we obtain
 \begin{equation}\label{eq6.27}
S_6 \le 2\d S_6' + \frac{\d N}{\log N} \ J_9,
 \end{equation}
with
 \begin{align}
J_9 &= \int_{I_7} \frac{dx}x \left(\frac 2{\frac 23 - 8\s - x} - \frac 2b\ \omega \left(\frac{1-\a_1}{b/2}\right) + \int_{\rho/2}^\rho \frac{dy}{y(1-x-y)}\right)\label{eq6.28}\\[2mm]
&< 0.727494\notag
 \end{align}
(we obtain this using \eqref{eq6.5}).  Since $S_0' = (1 + \l\e)$ $\frac N{2\log N}$, we only need to add up our integrals to obtain \cref{thm1} from \eqref{eq6.15}--\eqref{eq6.28}:
 \[
J_1 + J_2 + J_3 + J_4 + J_5 + J_6 + J_7 + J_9 < 0.994569 < 1 - \frac 1{200}.
 \]

\hskip .5in

 \end{document}